\documentclass[11pt]{amsart}

\usepackage{amsmath}
\usepackage{amssymb}
\usepackage{graphicx}

\newtheorem{theorem}{Theorem}[section]
\newtheorem{corollary}[theorem]{Corollary}
\newtheorem{lemma}[theorem]{Lemma}
\newtheorem{proposition}[theorem]{Proposition}
\theoremstyle{definition}
\newtheorem{definition}[theorem]{Definition}
\newtheorem{example}[theorem]{Example}
\newtheorem{remark}[theorem]{Remark}

\numberwithin{equation}{section}

\title[Existence results for strong vector equilibrium problems]
{Some enhanced existence results for strong vector equilibrium problems}

\author[A. Uderzo]{Amos Uderzo}
\address[A. Uderzo]{Dept. of Mathematics and Applications, University
of Milano - Bicocca, Milano, Italy}
\email{{\tt amos.uderzo@unimib.it}}



\date{\today}

\newcommand{\R}{\mathbb R}
\newcommand{\N}{\mathbb N}

\newcommand{\X}{\mathbb X}
\newcommand{\Y}{\mathbb Y}
\newcommand{\Uball}{{\mathbb B}}
\newcommand{\Usfer}{{\mathbb S}}

\newcommand{\dom}{{\rm dom}\, }

\newcommand{\nullv}{\mathbf{0}}

\newcommand{\wclco}{\overline{\rm conv}{\,}^*\, }

\newcommand{\cone}{{\rm cone}\, }

\newcommand{\inte}{{\rm int}\, }

\newcommand{\weakstar}{weak${}^*\, $ }

\newcommand{\Lin}{\mathcal{L}}
\newcommand{\Ph}{\mathcal{P}}

\newcommand{\SVE}{({\tt SVE})}
\newcommand{\VE}{({\tt VE})}
\newcommand{\WVE}{({\tt WVE})}
\newcommand{\Fmap}{F_{f,K}}
\newcommand{\Solv}{{\mathcal S}{\mathcal E}}

\newcommand{\parord}{\le_{{}_C}}

\newcommand{\ssinf}{|\nabla_K f|^{>}}

\newcommand{\Fsub}{\widehat{\partial}}
\newcommand{\Msub}{\partial_{\rm M}}

\newcommand{\dC}{d_C}
\newcommand{\nuka}{\nu_{+K}}
\newcommand{\Uplim}{{\rm Limsup}}

\newcommand{\stsl}[1]{|\nabla #1|}
\newcommand{\stslK}[1]{|\nabla_K #1|}

\newcommand{\dcone}[1]{{#1}^{{}^\ominus}}
\newcommand{\ndc}[1]{{#1}^{{}^\ominus}}

\newcommand{\ball}[2]{{\rm B}(#1, #2)}

\newcommand{\dist}[2]{{\rm dist}\left(#1,#2\right)}
\newcommand{\exc}[2]{{\rm exc}(#1,#2)}
\newcommand{\Ncone}[2]{{\rm N}(#1;#2)}
\newcommand{\FNcone}[2]{\widehat{\rm N}(#1;#2)}

\newcommand{\incr}[2]{{\rm inc}_C(#1;#2)}   
\newcommand{\Bderm}[2]{{\rm D}_{B}#1(#2)}
\newcommand{\Fderm}[2]{{\rm D}#1(#2)}
\newcommand{\Proj}[2]{\Pi(#1;#2)}

\newcommand{\Fder}[3]{{\rm D}#1(#2)#3}
\newcommand{\Bder}[3]{{\rm D}_{B}#1(#2;#3)}

\begin{document}

\begin{abstract}
This paper explores some sufficient conditions for the enhanced
solvability of strong vector equilibrium problems, which can be established
via a variational approach. Enhanced solvability here means existence
of solutions, which are strong with respect to the partial ordering,
complemented with inequalities estimating the distance from the solution
set (namely, error bounds). This kind of estimates plays a crucial role
in the tangential (first-order) approximation of the solution set
as well as in formulating optimality conditions for mathematical programming
with equilibrium constraints (MPEC).

The approach here followed characterizes solutions as zeros (or global
minimizers) of some merit functions associated to the original problem.
Thus, to achieve the main results the traditional employment of the KKM
theory is replaced by proper conditions on the slope of the merit functions.
In turn, to make such conditions verifiable, some tools of nonsmooth analysis
are exploited. As a result, several conditions for the enhanced solvability
of strong equilibrium problems are derived, which are expressed in terms
of generalized (Bouligand) derivatives, convex normals and various
(Fenchel and Mordukhovich) subdifferentials.
\end{abstract}

\maketitle


\section{Introduction}

Since more than two decades vector equilibrium problems are a topic
of active investigations, typically conducted by methods of nonlinear
and convex analysis.
The investigations exposed in the present paper consider vector
equilibrium problems, which are defined by a vector valued bifunction
taking values in a partially ordered space $\Y$, namely
$f:\X\times\X\longrightarrow\Y$, and a (nonempty) closed constraint set
$K\subseteq\X$. Throughout the paper, the relation $\parord$ partially ordering
$\Y$ is supposed to be induced in the standard way by a fixed closed
nontrivial convex cone $C\subseteq\Y$ (i.e. $\{\nullv\}\ne C\ne\Y$).

Similarly as in vector optimization, also for vector equilibrium problems
the solution concept is not unambiguously defined a priori. Given the
aforementioned data, by strong vector equilibrium the following problem
is meant:
$$
  \hbox{ find $\bar x\in K$ such that } f(\bar x,x)\in C,
  \quad\forall x\in K.   \leqno \SVE
$$
The solution set associated with problem $\SVE$ will be denoted
throughout the paper by $\Solv$. Sufficient conditions for its
nonemptiness and estimates for the distance from it are the main
theme of the present paper.

With the above data, one may also consider the different problem
$$
  \hbox{ find $\bar x\in K$ such that } f(\bar x,x)\not\in
  -C\backslash\{\nullv\},
  \quad\forall x\in K,   \leqno \VE
$$
called vector equilibrium problem. Furthermore,
if $\inte C\ne\varnothing$, it makes sense to consider the
so-called weak vector equilibrium problem,
meaning
$$
  \hbox{ find $\bar x\in K$ such that } f(\bar x,x)\not\in -\inte C,
  \quad\forall x\in K.   \leqno \WVE
$$
It is clear from the respective definitions that every solution to a
problem $\SVE$ is a fortiori a solution to the problems $\VE$ and $\WVE$,
defined by the same data (whence the terminology). Of course, elementary examples
show that the converse is not true. If, in particular, $\Y=\R$ and
$C=[0,+\infty)$, then $\SVE$, $\VE$ and $\WVE$ collapse to the same problem,
namely what is called equilibrium problem after Blum and Oettli.
By their seminal paper \cite{BluOet94}, they definitely contribute
to popularize this kind of problem, in stressing its unifying feature and
undertaking a thorough study of it.
In fact, equilibrium problem revealed to be a convenient format
to treat in a unified framework various problems which are relevant
in operations research and mathematical programming, such as single and
multicriteria optimization problems, saddle point problems, complementarity
problems, variational inequalities, fixed point problems, Nash equilibrium
problems. In a similar manner, vector equilibrium problems
provide a format able to subsume vector optimization problems,
vector complementarity problems and vector variational inequalities
(see \cite{AnOeSc97,Ansa00,Gong06} and references therein).
For this reason, in the last two decades vector equilibrium problems
became the subject of many investigations. As it is reasonable,
within the fast growing literature in this area, a remarkable
amount of research work focussed on solution existence and related issues
(see, for instance, \cite{AnKoYa01,AnOeSc96,AnOeSc97,Ansa00,BiHaSc97,Gong06}).
One of the main techniques of analysis in this context consists in
adapting to the vector case the approach due to Ky Fan, originally
proposed for scalar equilibrium problems (see, for instance,
\cite{BiCaPaPa13}). To achieve the nonemptiness
of the solution set, regarded as the intersection of a proper family
of sets, this approach leads to apply the Knaster-Kuratowski-Mazurkiewicz
theorem (a.k.a. Three Polish theorem) or some variant of it
(see \cite{BiHaSc97,CeMaYa08}). Other approaches to solution existence
specific for strong vector equilibrium problems rely on different techniques,
such as the employment of separation theorems for convex sets (see \cite{Gong06})
or the Kakutani fixed point theorem (see \cite{AnOeSc96}).

The aim of the present paper is to enhance the study of solvability
for strong vector equilibrium problems by complementing results about solution
existence with inequalities estimating the distance from the solution set.
These error bounds for $\SVE$ provide useful quantitative information
on the set of solutions, which may be exploited in various contexts of
application. For example, it is well known that error bounds enable to describe
the local geometry of the solution set of a problem through its
tangential (first-order) approximation.
Moreover, error bounds are known to be connected with the metric subregularity
and calmness properties of set-valued mappings (see, for instance, \cite{FaHeKrOu10}).
Thus, according to a recognized approach of analysis, error bounds reveal
to be an essential tool for establishing optimality conditions via
penalization techniques. More precisely, the estimates presented
in this paper should be propaedeutic in order for deriving optimality conditions
for MPEC, where equilibrium constraints take the form of strong vector
equilibrium problems (see \cite[Section 5.2.3]{Mord06b}).
Furthermore, error bounds turns out to play an important role in
the convergence theory of numerical methods.

It is plain to see that a problem $\SVE$ can be reformulated as a
set-valued inclusion. If a set-valued mapping $\Fmap:\X\rightrightarrows\Y$
is defined as
$$
 \Fmap(x)=f(x,K)=\{y\in\Y:\ y=f(x,z),\ z\in K\},
$$
then problem $\SVE$ becomes
$$
  \hbox{ find $\bar x\in K$ such that } \Fmap(\bar x)\subseteq C.
$$
Various elements for a solution analysis of the latter problem
have been recently proposed in a series of papers \cite{Uder19,Uder20,Uder21,Uder22},
where several theoretical aspects of the solution behaviour (including
existence and stability issues) have been investigated.
The work presented in this paper can be viewed as an
attempt to specialize that line of research to the context of
vector equilibria.

Following a variational approach, the first step consists
in introducing some functional characterizations of $\Solv$. This is done by
associating to a problem $\SVE$ a sort of merit function $\nu:\X\longrightarrow
[0,+\infty]$, which is defined as
\begin{eqnarray}    \label{eq:defmeritf}
  \nu(x) &=&\exc{\Fmap(x)}{C}=\sup_{y\in\Fmap(x)}\dist{y}{C} \\
  &=&\sup_{z\in K}\dist{f(x,z)}{C}.  \nonumber
\end{eqnarray}
In order to embed also the constraining set $K$, it is useful
to consider as well the function $\nuka:\X\longrightarrow [0,+\infty]$,
given by
\begin{equation}    \label{eq:defnuka}
  \nuka(x)=\nu(x)+\dist{x}{K}.
\end{equation}
Such merit functions, incorporating all problem data, enable
to reduce strong vector equilibria to zeros (or global
minimizers) of a functional.

\begin{remark}     \label{rem:funchar}
Since $C$ is closed, one sees that $\bar x\in\Solv$ iff $\bar x\in K$ and
$\nu(\bar x)=0$. Equivalently, it holds
$$
  \Solv=\nu^{-1}((-\infty,0])\cap K=\nu^{-1}(0)\cap K.
$$
Analogously, since $K$ is closed, one sees that $\bar x\in\Solv$ iff
$\nuka(\bar x)=0$, namely
$$
  \Solv=\nuka^{-1}((-\infty,0])=\nuka^{-1}(0).
$$
The above functional characterizations of $\Solv$ enable also to clarify
at once some of its structural vector-topological properties. Namely,
whenever $\nu$ and $K$ are convex, $\Solv$ is convex (possibly empty).
Whenever $\nu$ is l.s.c. on $\X$ ($K$ being closed), $\Solv$
is closed (possibly empty).
A sufficient condition for the latter property of $\nu$
to hold is, for instance, that each function $x\mapsto f(x,z)$
is continuous on $\X$, for every $z\in K$. Indeed, in
such an event, as the distance function $y\mapsto \dist{y}{C}$ is Lipschitz
continuous on $\Y$ and therefore each function $x\mapsto \dist{f(x,z)}{C}$
is continuous on $\X$, then, according to $(\ref{eq:defmeritf})$, $\nu$ can be
expressed as an upper envelope of continuous functions on $\X$.
\end{remark}

The contents of the paper are organized as follows. Section \ref{Sect:2}
contains some preliminary technicalities dealing with the basic tools
of analysis. Other advanced tools are recalled in the subsequent section,
contextually to their use.
Section \ref{Sect:3} contains the main results of the paper arranged
in two subsections:
in the first one, conditions for the enhanced existence are presented,
which rely on metric increase behaviour of the involved bifunctions,
whereas in the second one some conditions are expressed in terms of
several subdifferentials.
Section \ref{Sect:4} is reserved for concluding remarks.

Below, let us introduce the basic notations employed in the paper.
The acronyms l.s.c., u.s.c. and p.h. stand for lower semicontinuous, upper
semicontinuous and positively homogeneous, respectively.
In a metric space setting, the closed  ball centered at an element $x$,
with radius $r\ge 0$, is denoted by $\ball{x}{r}$. In particular,
in a Banach space, $\Uball=\ball{\nullv}{1}$, whereas $\Usfer$ stands
for the unit sphere.
The distance of a point $x$ from $S$ is denoted by $\dist{x}{S}$, with
the convention that $\dist{x}{\varnothing}=+\infty$. The function
$x\mapsto\dist{x}{S}$ is sometimes indicated by $d_S$, if convenient.
$(\X,\|\cdot\|)$ and $(\Y,\|\cdot\|)$ denote real Banach spaces, whose
null vector is indicated by $\nullv$.
Given a subset $S$ of a Banach space, $\inte S$ denotes its interior, whereas
$\cone S$ its conical hull.
By $\Ph(\X,\Y)$ the Banach space of all continuous p.h. operators
acting between $\X$ and $\Y$ is denoted, equipped with the operator norm
$\|h\|_\Ph=\sup_{u\in\Usfer}\|h(u)\|$, $h\in\Ph(\X,\Y)$.
$\Lin(\X,\Y)$ denotes its subspace of all bounded linear
operators and, if $\Lambda\in\Lin(\X,\Y)$, $\Lambda^*\in\Lin(\Y^*,\X^*)$
indicates the adjoint operator to $\Lambda$.
In particular, $\X^*=\Lin(\X,\R)$ stands for the dual
space of $\X^*$, in which case $\|\cdot\|_\Ph$ is simply marked by $\|\cdot\|$.
The null vector, the unit ball and the unit sphere in a dual space
will be marked by $\nullv^*$, $\Uball^*$, and $\Usfer^*$, respectively.
The duality pairing of a Banach space with its dual will be denoted
by $\langle\cdot,\cdot\rangle$. If $S$ is a subset of a dual space,
$\wclco S$ stands for its convex closure with respect to the
\weakstar topology. Whenever $C\subseteq\Y$ is a cone, by $\dcone{C}
=\{y^*\in\Y^*:\ \langle y^*,y\rangle\le 0,\quad\forall y\in C\}$
its negative dual cone is denoted.
Given a function $\varphi:\X\longrightarrow\R\cup\{\pm\infty\}$,
by $[\varphi\le 0]=\varphi^{-1}([-\infty,0])$ its $0$-sublevel set is
denoted, whereas $[\varphi>0]=\varphi^{-1}((0,+\infty])$ denotes
the strict $0$-superlevel set of $\varphi$.
The symbol $\dom\varphi=\varphi^{-1}(\R)$ indicates the domain of
the function $\varphi$, while $\partial\varphi(x)$ the subdifferential
of $\varphi$ at $x$ in the sense of convex analysis (a.k.a. Fenchel
subdifferential), with the convention $\partial\varphi(x)=\varnothing$
if $x\not\in\dom\varphi$. The normal cone to a set $S$ at $x$ in the
sense of convex analysis is denoted by $\Ncone{x}{S}$.

\vskip1cm



\section{Preliminary tools of analysis}    \label{Sect:2}

A first group of technical preliminaries relate to semicontinuity
and convexity properties of the merit functions $\nu$ and $\nuka$.
Recall that, according to \cite{Luc89}, given a closed, convex cone $C$,
a mapping $g:\X\longrightarrow\Y$ between Banach spaces
is said to be $C$-l.s.c. (resp. $C$-u.s.c.) at $x_0\in\X$ if for any
neighbourhood $V$ of $g(x_0)$ there exists a neighbourhood $U$ of
$x_0$ in $\X$ such that
$$
  g(x)\in V+C, \qquad \hbox{(resp. $g(x)\in V-C$)}
  \qquad \forall x\in U.
$$
Clearly, continuous mappings are both $C$-l.s.c. and $C$-u.s.c.,
whereas $C$-semicontinuity does not imply continuity, in general.

\begin{lemma}    \label{lem:Cusclsc}
If $g:\X\longrightarrow\Y$ is $C$-u.s.c. at $x_0\in\X$, then
function $\dC\circ g$ is l.s.c. at $x_0$.
\end{lemma}

\begin{proof}
In the current Banach space setting the property of semicontinuity
can be proven to hold by sequential arguments.
So, fix arbitrarily $\epsilon>0$ and a sequence $(x_n)_n$, with $x_n\to x_0$
as $n\to\infty$. By the $C$-upper semicontinuity of $g$ at $x_0$,
corresponding to $V=\ball{g(x_0)}{\epsilon}$ there exists
$n_\epsilon\in\N$ such that
$$
    g(x_n)\in\ball{g(x_0)}{\epsilon}-C,\quad
    \forall n\ge n_\epsilon.
$$
This means that there are $v_n\in\epsilon\Uball$ and $c_n\in C$
such that it is possible to write
\begin{equation}     \label{eq:BCrep}
   g(x_n)=g(x_0)+v_n-c_n,\quad\forall n\ge n_\epsilon.
\end{equation}
Thus, as it is $c_n+C\subseteq C$ and hence $C\subseteq C-c_n$,
on account of the representation $(\ref{eq:BCrep})$ one
obtains
\begin{eqnarray*}
  (\dC\circ g)(x_0) &=&  \inf_{c\in C}\|g(x_0)-c\|=
  \inf_{c\in C}\|g(x_n)-v_n+c_n-c\| \\
  & \le & \inf_{c\in C}\|g(x_n)-(c-c_n)\|+\|v_n\|  \\
  &=& \inf_{y\in C-c_n}\|g(x_n)-y\|+\|v_n\|
  \le  \inf_{c\in C}\|g(x_n)-c\|+\epsilon \\
  &=& (\dC\circ g)(x_n)+\epsilon,\quad\forall n\ge n_\epsilon.
\end{eqnarray*}
From these inequalities it follows
$$
  \liminf_{n\to\infty}(\dC\circ g)(x_n)\ge(\dC\circ g)(x_0)
  -\epsilon,
$$
which by arbitrariness of $\epsilon$ and $(x_n)_n$
proves the assertion in the thesis.
\end{proof}

\begin{remark}    \label{rem:nulsc}
For the purposes of the present analysis it is useful to note
that, as a straightforward consequence of Lemma \ref{lem:Cusclsc},
one can deduce that if each function $x\mapsto f(x,z)$ is
$C$-u.s.c. on $K$ for every $z\in K$, then function $\nu(x)=
\sup_{z\in K}(\dC\circ f)(x,z)=\sup_{z\in K}\dist{f(x,z)}{C}$
is l.s.c. on $K$, as an upper envelope of functions being
l.s.c. on $K$. Since $d_K$ is (Lipschitz) continuous on $\X$,
the same is true for $\nuka$.
\end{remark}

In order to formulate the next lemma it is convenient to recall
that a subset $C$ of a Banach space partially
ordered by a cone $C$ is said to be $C$-bounded if there exists
a constant $m\ge 0$ such that $S\backslash C\subseteq m\Uball$.

\begin{lemma}   \label{lem:abvbound}
Let $f:\X\times\X\longrightarrow\Y$ be a given bifunction, $K\subseteq\X$
and $x_0\in K$. If the set $f(x_0,K)$ is $C$-bounded, then $x_0\in\dom\nu$.
If, in addition, each function $x\mapsto f(x,z)$ is $C$-l.s.c. at $x_0$
uniformly in $z\in K$, then $\nu$ is bounded from above in a neighbourhood
of $x_0$ and hence $x_0\in\inte\dom\nu$.
\end{lemma}

\begin{proof}
By hypothesis, for some $m>0$ it holds
$$
  f(x_0,K)\backslash C\subseteq m\Uball,
$$
whence it follows
$$
  \nu(x_0)=\sup_{y\in f(x_0,K)\backslash C}\dist{y}{C}\le
  \sup_{y\in m\Uball}\dist{y}{C}\le\sup_{y\in m\Uball}\|y\|=m<+\infty.
$$
Thus, $x_0\in\dom\nu$.

According to the additional hypothesis, corresponding to $V=l\Uball$
there exists a neighbourhood $U$ of $x_0$ (not depending on $z\in K$)
such that
$$
  f(x,z)\in\ball{f(x_0,z)}{l}+C,\quad\forall x\in U,\ \forall z\in K.
$$
This amounts to say that, for any $z\in K$ and $x\in U$, there exist
$v\in\ball{f(x_0,z)}{l}$ and $c\in C$ (both depending on $x\in U$
and $z\in K$), such that $f(x,z)=v+c$.
Thus, recalling that $C+C\subseteq C$, for any $z\in K$
one finds
\begin{eqnarray*}
   \dist{f(x,z)}{C}&=&\dist{v+c}{C}=\inf_{y\in C}\|v+c-y\| \\
   &\le &\inf_{c_1\in C}\inf_{c_2\in C}\|v+c-(c_1+c_2)\| \\
   &\le &\inf_{c_1\in C}\inf_{c_2\in C}[\|v-c_1\|+\|c-c_2\|]
   \le \inf_{c_1\in C}\|v-c_1\|  \\
   &=&\dist{v}{C},\quad\forall x\in U.
\end{eqnarray*}
Therefore, it follows
\begin{eqnarray*}
  \sup_{x\in U}\dist{f(x,z)}{C} &\le& \|v-f(x_0,z)\|+
  \dist{f(x_0,z)}{C}  \\
  &\le & l+\dist{f(x_0,z)}{C},\quad\forall z\in K.
\end{eqnarray*}
Consequently, one obtains
\begin{eqnarray*}
   \sup_{x\in U}\nu(x) &=& \sup_{x\in U}\sup_{z\in K}\dist{f(x,z)}{C}
   =\sup_{z\in K}\sup_{x\in U}\dist{f(x,z)}{C} \\
   &\le &\sup_{z\in K} [l+\dist{f(x_0,z)}{C}]=
   l+\nu(x_0)<+\infty,
\end{eqnarray*}
which means that $\nu$ is bounded from above on $U$.
\end{proof}

Another key assumption for the present approach is $C$-concavity
for mappings. Following \cite{FreKas99}, a mapping $g:\X\longrightarrow\Y$
between Banach spaces is said to be $C$-concave on the convex set
$K\subseteq\X$ if for every $x_1,\, x_2\in K$ and for every $t\in [0,1]$
it is true that
$$
 tg(x_1)+(1-t)g(x_2)\parord g(tx_1+(1-t)x_2).
$$

\begin{example}    \label{ex:Cconcavemap}
(i) It is readily seen that if a mapping $g:\R^n\longrightarrow\R^m$
is defined by components $g_i:\R^n\longrightarrow\R$, each of which
is concave on a convex set $K\subseteq\R^n$, then $g$ is $\R^m_+$-concave
on $K$.

(ii) A remarkable class of $C$-concave mappings is the subclass
of $\Ph(\X,\Y)$ formed by the superlinear operators taking values
in a Kantorovich space $\Y$ (i.e. a Dedekind complete normed vector
lattice), partially ordered by a cone $C$.
Following \cite{Rubi77}, a mapping $h\in\Ph(\X,\Y)$ is said to
be superlinear if $h(x)+h(z)\parord h(x+z)$, for every $x,\, z\in\X$.
It is well known that, as a consequence of the
Hahn-Banach-Kantorovich theorem (see \cite{KanAki82}), any superlinear
operator admits the following infimal representation
$$
  h(x)=\min{}_{\parord}\{\Lambda x:\ \Lambda\in\overline{\partial}h\},
$$
where $\overline{\partial}h=\{\Lambda\in\Lin(\X,\Y):\ h(x)\parord\Lambda x,
\ \forall x\in\X\}$ and $\min_{\parord}S$ denotes the smallest
element of a set $S\subseteq\Y$ with respect to the partial order $\parord$.
\end{example}

\begin{lemma}    \label{lem:Cconcconv}
If $g:\X\longrightarrow\Y$ is $C$-concave on the convex set
$K\subseteq\X$, then function $\dC\circ g$ is convex on $K$.
\end{lemma}

\begin{proof}
Take arbitrary $x_1,\, x_2\in K$ and $t\in [0,1]$. Owing to the
$C$-concavity of $g$ on $K$, one has
$$
  c_0=g(tx_1+(1-t)x_2)-[tg(x_1)+(1-t)g(x_2)]\in C.
$$
Since it is $c_0+C\subseteq C$, from the above inclusion
one obtains
\begin{eqnarray*}
   (\dC\circ g)(tx_1+(1-t)x_2) &=& \inf_{c\in C}\|g(tx_1+(1-t)x_2)-c\| \\
   &\le &\inf_{c\in c_0+C}\|g(tx_1+(1-t)x_2)-c\| \\
   &\le & \inf_{c\in C}\|g(tx_1+(1-t)x_2)-\{g(tx_1+(1-t)x_2) \\
   & & -[tg(x_1)+(1-t)g(x_2)]+c\}\|  \\
   &=& \dist{tg(x_1)+(1-t)g(x_2)}{C}.
\end{eqnarray*}
By recalling that the function $y\mapsto\dist{y}{C}$ is sublinear
on $\Y$ as $C$ is a convex cone, then from the above inequalities
it follows
\begin{eqnarray*}
  (\dC\circ g)(tx_1+(1-t)x_2) &\le& t\dist{g(x_1)}{C}+(1-t)\dist{g(x_2)}{C} \\
  &=& t(\dC\circ g)(x_1)+(1-t)(\dC\circ g)(x_2),
\end{eqnarray*}
which, by arbitrariness of $x_1,\, x_2\in K$ and $t\in [0,1]$
completes the proof.
\end{proof}

The next lemma singles out a sufficient condition on the problem
data of $\SVE$ in order for $\nu$ to be a (proper) convex and continuous
function.

\begin{lemma}    \label{lem:convcontmerfun}
Let $f:\X\times\X\longrightarrow\Y$ be a given bifunction and
let $K\subseteq\X$ be a convex set. Suppose that:

\begin{itemize}

\item[(i)] each function $x\mapsto f(x,z)$ is $C$-concave on $K$,
for every $z\in K$;

\item[(ii)] there exists $x_0\in K$ such that $f(x_0,K)$
is $C$-bounded and function $x\mapsto f(x,z)$ is $C$-l.s.c.
at $x_0$, uniformly in $z\in K$.

\end{itemize}
Then, function $\nu:K\longrightarrow [0,+\infty]$ is convex
and continuous on $\inte\dom\nu\ne\varnothing$.
\end{lemma}

\begin{proof}
According to Lemma \ref{lem:abvbound}, by virtue of hypothesis
(ii) $x_0\in\inte\dom\nu$ and $\nu$ turns out to be bounded
from above in a neighbourhood of $x_0$.
According to Lemma \ref{lem:Cconcconv}, each function $x\mapsto\dist{f(x,z)}{C}$
is convex on $K$ and therefore so is function $\nu$, which
can be regarded as an upper envelope of functions $\dist{f(\cdot,z)}{C}$
over $z\in K$.
As a convex function, which is bounded from above in a neighbourhood
of $x_0$, $\nu$ must be continuous on the interior of its domain,
in the light of a well-known result in convex analysis (see, for
instance, \cite[Theorem 2.2.9]{Zali02}).
\end{proof}

\begin{remark}    \label{rem:conconvnuka}
In view of the subsequent analysis, it is convenient to
notice that, under the hypotheses of Lemma \ref{lem:convcontmerfun},
also function $\nuka$ turns out to be convex and continuous,
with $\dom\nuka=\dom\nu\ne\varnothing$.
\end{remark}

A characteristic feature of the main results established in the next section
is to provide, along with solution existence,
quantitative (metric) information about the
solution set to strong vector equilibrium problems.
Following a standard technique of variational analysis,
estimates for the distance from the solution set to $\SVE$
will be investigated by means of a metric slope of the merit
functions.
Given a function $\varphi:\X\longrightarrow\R\cup\{\pm\infty\}$
defined on a Banach space, a closed set $K\subseteq\X$ and
$x\in K\cap\dom\varphi$, the nonnegative value
$$
  \stslK{\varphi}(x)=\left\{\begin{array}{ll}
  0, & \hbox{ if $x$ is a local minimizer} \\
     & \hbox{ of $\varphi$ subject to $x\in K$},  \\
  \displaystyle\limsup_{u\stackrel{K}{\to} x}
  {\varphi(x)-\varphi(u)\over\|x-u\|}, & \hbox{ otherwise, }
  \end{array}\right.
$$
where $u\stackrel{K}{\to} x$ means $u\to x$ while $u\in K$,
represents the slope of $\varphi$ at $x$ restricted to $K$.
It may be regarded as a restricted version of the well-known
notion of (strong) slope of a function $\varphi:X\longrightarrow
\R\cup\{\pm\infty\}$ defined on a metric space $(X,d)$,
denoted here by $\stsl{\varphi}(x)$, which was introduced in
\cite{DeMaTo80} and subsequently employed in various contexts of variational
analysis (see, among the others, \cite{Aze03,AzeCor04,Ioff00,Peno13}).

\begin{remark}     \label{rem:stslstslK}
(i) Directly from the above definition, one sees that in general
it holds
$$
  \stslK{\varphi}(x)\le\stsl{\varphi}(x),
$$
while, whenever it is $x\in\inte K$, one has $\stslK{\varphi}(x)=
\stsl{\varphi}(x)$.

(ii) Remember that, whenever $\varphi$ is Fr\'echet differentiable
at $x\in\dom\varphi$, it holds $\stsl{\varphi}(x)=\|\Fderm{\varphi}{x}\|$.

(iii) Whenever $\varphi$ is a proper, convex and l.s.c. function,
one has $\stsl{\varphi}(x)=\dist{\nullv^*}{\partial\varphi(x)}$
(see, for instance, \cite[Proposition 3.1]{AzeCor04} or
\cite[Theorem 5(ii)]{FaHeKrOu10}).
\end{remark}

The seminal condition for the solution existence of $\SVE$
presented in Section \ref{Sect:3} will be
formulated in terms of the following crucial value associated
through $\nu$ with a problem $\SVE$:
$$
  \ssinf=\inf_{x\in [\nu>0]\cap K}\stslK{\nu}(x).
$$

It is well known from variational analysis (see, for instance,
\cite{Aze03,AzeCor04,Ioff00,Peno13,Uder19})
that distances from sublevel sets of a l.s.c. function can be
estimated in terms of its slope, while a certain positivity behaviour
of the slope allows one to establish solvability of inequalities.
All of this leads to investigate metric behaviours of the bifunction
$f$, which ensure these conditions to hold. The property
defined below pursues this purpose. It is a uniform variant
of a conceptual tool that was already considered in \cite{Uder19},
in connection with the analysis of solvability and stability properties
of set-valued inclusions (see also \cite{Uder21}).

\begin{definition}   \label{def:metincrunif}
Let $f:\X\times\X\longrightarrow\Y$ be a bifunction and let
$K\subseteq\X$  and $S\subseteq\X$ be given sets. Then, $f$ is said to be
{\it metrically $C$-increasing} on the set $S$, {\it uniformly in} $z\in K$,
if there exists $\alpha>1$ such that
for every $x_0\in S$ there is $\delta_0>0$ (not depending on $z$)
such that

\begin{eqnarray}   \label{in:Cincr}
 & \forall r\in (0,\delta_0)\ \exists x\in\ball{x_0}{r}\cap S:  \nonumber \\
 & \ball{f(x,z)}{\alpha r}\subseteq\ball{f(x_0,z)+C}{r},
 \quad\forall z\in K.
\end{eqnarray}
The value
$$
  \incr{f}{S}=\sup\{\alpha>1:\ \hbox{ condition $(\ref{in:Cincr})$
  holds}\}
$$
is called {\it exact bound of uniform $C$-metric increase} of $f$ over $S$.
\end{definition}

In view of the next proposition, it is useful to recall some
relations involving the behaviour of the excess of sets that
will be exploited in its proof.
Let $C\subseteq\Y$ be a closed, convex cone and let $S\subseteq\Y$
a set with $S\not\subseteq C$ (hence, nonempty). Then, it holds
\begin{itemize}

\item[$(p_1)$] $\exc{\ball{S}{r}}{C}=\exc{S}{C}+r,\quad\forall r>0$
(see \cite[Lemma 2.2]{Uder19});

\item[$(p_2)$] $\exc{S+C}{C}=\exc{S}{C}$ (see \cite[Remark 2.1]{Uder19}).

\end{itemize}

The next proposition shows that the metric $C$-increase property is able
to capture a behavior of $f$, which
is useful in providing estimates from below of $\ssinf$ that are convenient
to the present approach.

\begin{proposition}   \label{lem:Cincrstsl}
Let $f:\X\times\X\longrightarrow\Y$ be a given bifunction and
let $K\subseteq\X$. Suppose that:
\begin{itemize}

\item[(i)] function $\nu$, associated with $f$ and $K$ as in
$(\ref{eq:defmeritf})$, is l.s.c. on $K$;

\item[(ii)] $f(\cdot,z)$ is metrically $C$-increasing on $K\cap[\nu>0]$,
uniformly in $z\in K$.

\end{itemize}
Then, it holds
\begin{equation}   \label{in:ssinfincrest}
  \ssinf\ge\incr{f}{K\cap[\nu>0]}-1.
\end{equation}
\end{proposition}

\begin{proof}
Take arbitrary $\alpha\in (1,\incr{f}{K\cap[\nu>0]})$ and
$x_0\in K\cap[\nu>0]$. Since $\nu$ is l.s.c. at $x_0$,
there exists $\delta>0$ such that $\ball{x_0}{\delta}\cap K
\subseteq [\nu>0]$. By virtue of hypothesis (ii),
for any $r\in (0,\delta_0)$, where $\delta_0\in (0,\delta)$
is as in Definition \ref{def:metincrunif}, there exists
$x_r\in\ball{x_0}{r}\cap K\cap [\nu>0]$ such that
$\ball{f(x_r,z)}{\alpha r}\subseteq\ball{f(x_0,z)+C}{r}$,
for every $z\in K$. This implies
$$
  \ball{f(x_r,z)}{\alpha r}\subseteq\ball{f(x_0,K)+C}{r},
  \quad\forall z\in K.
$$
So, fix any $\tilde{\alpha}\in (1,\alpha)$. Since a fortiori it is
$$
  \ball{f(x_r,z)}{\tilde{\alpha} r}\subseteq\ball{f(x_0,K)+C}{r},
  \quad\forall z\in K,
$$
one sees from the last inclusion that
$\ball{f(x_r,K)}{\tilde{\alpha} r}\subseteq\ball{f(x_0,K)+C}{r}$,
whence it is possible to deduce that
$\ball{f(x_r,K)}{\alpha r}\subseteq\ball{f(x_0,K)+C}{r}$,
because the set $\ball{f(x_0,K)+C}{r}$ is closed and
$\tilde{\alpha}$ arbitrary in $(1,\alpha)$.
Notice that it must be $x_r\ne x_0$
because $\alpha>1$.
Thus, since $\Fmap(x_r)\not\subseteq C$, on account of the above
recalled relations $(p_1)$ and $(p_2)$, one obtains
\begin{eqnarray*}
  \nu(x_r) &=& \exc{\ball{\Fmap(x_r)}{\alpha r}}{C}-\alpha r \\
  &\le&  \exc{\ball{\Fmap(x_0)+C}{r}}{C}-\alpha r \\
  &\le &\nu(x_0)+r-\alpha r.
\end{eqnarray*}
As $x_r$ belongs to $\ball{x_0}{r}\cap K\cap [\nu>0]$, one finds
$$
  \nu(x_0)-\nu(x_r)\ge (\alpha-1)r\ge (\alpha-1)\|x_r-x_0\|.
$$
By arbitrariness of $r\in (0,\delta_0)$, it results in
$$
  \stslK{\nu}(x_0)\ge\alpha-1.
$$
Since the last inequality is true all over $K\cap[\nu>0]$ by the
arbitrariness of $x_0$, then one can conclude that $\ssinf\ge\alpha-1$.
From this inequality one achieves the estimate
in $(\ref{in:ssinfincrest})$ by taking into account the arbitrariness
of $\alpha\in (1,\incr{f}{K\cap[\nu>0]})$.
\end{proof}



\section{Enhanced solution existence for strong vector
equilibrium problems}    \label{Sect:3}


The present approach to the strong solvability of vector
equilibrium problems starts with a general result which relies
on the following specialization of \cite[Theorem 1.10]{Aze03}.

\begin{proposition}    \label{pro:Aze03}
Let $(X,d)$ be a complete metric space and let $\varphi:X\longrightarrow
\R\cup\{+\infty\}$ be a l.s.c. function. Assume that
$[\varphi<+\infty]\ne\varnothing$ and that
$$
  \tau=\inf_{x\in[0<\varphi<+\infty]}\stsl{\varphi}(x)>0.
$$
Then, it is $[\varphi\le 0]\ne\varnothing$ and
\begin{equation}    \label{in:erbovarphi}
  \dist{x}{[\varphi\le 0]}\le {\max\{\varphi(x),0\}\over\tau},
  \quad\forall x\in [\varphi<+\infty].
\end{equation}
\end{proposition}

\begin{proof}
The thesis follows directly from the assertion $a)$ in
\cite[Theorem 1.10]{Aze03}, with the choice $\alpha=\gamma=0$
and $\beta=+\infty$ for the parameters appearing in its statement.
Notice, in particular, that the nonemptiness of $[\varphi\le 0]$
comes as a consequence of inequality $(\ref{in:erbovarphi})$,
whose left-side term must be a real number for every $x\in
 [\varphi<+\infty]$.
\end{proof}

A first enhanced existence result for strong equilibrium problems
can be established in terms of constructions described in Section
\ref{Sect:2} as follows.

\begin{theorem}    \label{thm:SEexist}
With reference to a problem \SVE, suppose that:
\begin{itemize}

\item[(i)] each function $x\mapsto f(x,z)$ is $C$-u.s.c. on $K$,
for every $z\in K$;

\item[(ii)] there exists $x_0\in K$ such that $f(x_0,K)$
is $C$-bounded;

\item[(iii)] it is $\displaystyle\ssinf>0$.

\end{itemize}
Then, $\Solv$ is nonempty and closed and the
following estimate holds
\begin{equation}    \label{in:erbossinf}
    \dist{x}{\Solv}\le {\nu(x)\over\ssinf},
    \quad\forall x\in K.
\end{equation}
\end{theorem}

\begin{proof}
If $K\cap [\nu>0]=\varnothing$ it means that $\Solv=K$,
so all the assertions in the thesis become trivially true.
Assume henceforth that $K\cap [\nu>0]\ne\varnothing$.
In the light of Remark \ref{rem:nulsc}, under the made assumptions
function $\nu:K\longrightarrow [0,+\infty]$ turns out to be
l.s.c. on $K$ and, by virtue of Lemma \ref{lem:abvbound},
it is $x_0\in [\nu<+\infty]\ne\varnothing$.
As $K$ is closed, the metric space $(K,d)$, where $d$
is the metric induced by the norm of $\X$ on $K$, is
complete.
Thus it is possible to invoke Proposition \ref{pro:Aze03},
in such a way to get, in consideration of the functional characterization
of $\Solv$ (remember Remark \ref{rem:funchar}) $\Solv\ne\varnothing$. By taking
into account that $\nu$ takes nonnegative values only, from
inequality $(\ref{in:erbovarphi})$ one can reach the estimate
in the thesis for every $x\in K\cap[\nu<+\infty]$. The extension
of the validity of inequality $(\ref{in:erbossinf})$ in such a
way to include also $x\in K\cap[\nu=+\infty]$ is obvious.
As for the closedness of $\Solv$, it comes as an immediate
consequence of the lower semicontinuity property of $\nu$ on $K$.
This completes the proof.
\end{proof}

As it happens for existence and error bound results related
to several problems, which can be achieved by following
the present variational approach (see, among the others,
\cite{Uder19,Uder21}), Theorem \ref{thm:SEexist} provides a
sufficient condition for solvability, which
generally fails to be also necessary. The example below
aims at illustrating this fact.

\begin{example}
Let $\X=\Y=\R^2$ be equipped with its standard Euclidean
space structure, $C=\R^2_+$, $K=-\R^2_+$ and let $f:\R^2\times\R^2
\longrightarrow\R^2$ be defined by
\begin{equation}    \label{eq:anexpf}
  f(x,z)=\left(\begin{array}{c}
                  -x_1^2+e^{-\|z\|} \\
                  -x_2^2+{1\over\|z\|+1}
                \end{array}\right),\qquad
   x=(x_1,x_2),\, z=(z_1,z_2)\in\R^2.
\end{equation}
It is evident that $\bar x=\nullv$ is a solution to the problem
$\SVE$ defined by the given data. More precisely, it is
$\Solv=\{\nullv\}$. Indeed, if $\hat x\in -\R^2_+\backslash\{\nullv\}$
it must be $\min\{\hat x_1\,\,\hat x_2\}<0$. According to the
expression in $(\ref{eq:anexpf})$, if taking $z_k=(-k,0)\in-\R^2_+$
for every $k\in\N$, one finds
$$
    f(\hat x,z_k)=\left(\begin{array}{c}
                  -\hat x_1^2+e^{-k} \\
                  -\hat x_2^2+{1\over k+1}
                \end{array}\right)
        \stackrel{\small k\to\infty}{\longrightarrow}
        \left(\begin{array}{c}
                  -\hat x_1^2 \\
                  -\hat x_2^2
                \end{array}\right) \not\in\R^2_+,
$$
and hence, as $\R^2_+$ is closed, for some $z_k\in -\R^2_+$,
it must be true that $f(\hat x,z_k)\not\in\R^2_+$,
so $\hat x\not\in\Solv$.
Since $f$ is continuous on $\R^2$, in particular each function
$x\mapsto f(x,z)$ is $C$-u.s.c. on $-\R^2_+$, for every $z\in
-\R^2_+$. From being $e^{-\|z\|}$, $(\|z\|+1)^{-1}\in (0,1]$ for every
$z\in -\R^2_+$, one deduces that
$$
  \Fmap(x)=f(x,-\R^2_+)\subseteq  \left(\begin{array}{c}
                  -x_1^2 \\
                  -x_2^2
                \end{array}\right)+([0,1]\times [0,1]),
                \quad\forall x\in -\R^2_+,
$$
which shows that the set $f(x,-\R^2_+)$ is bounded (and hence, a fortiori,
$\R^2_+$-bounded) for every $x\in -\R^2_+$. Whereas hypotheses
(i) and (ii) of Theorem \ref{thm:SEexist} happen to be satisfied,
hypothesis (iii) does not. Indeed, for the problem under consideration
the merit function $\nu:\R^2\longrightarrow[0,+\infty)$ is
clearly given by the expression
$$
  \nu(x)=\left\|\left(
  \begin{array}{c}
                  -x_1^2 \\
                  -x_2^2
                \end{array}\right)\right\|=\sqrt{x_1^4+x_2^4},
                \quad\forall x\in -\R^2_+.
$$
By taking into account what noticed in Remark \ref{rem:stslstslK},
one obtains in particular
$$
  \stslK{\nu}(x)=\stsl{\nu}(x)=\|\Fderm{\nu}{x}\|,
  \quad\forall x\in\inte (-\R^2_+)=-\inte\R^2_+.
$$
Then, elementary calculations lead to find
$$
  \|\Fderm{\nu}{x}\|=\left\|\left(
  \begin{array}{c}
                  {2x_1^3\over\sqrt{x_1^4+x_2^4}} \\
                  {2x_2^3\over\sqrt{x_1^4+x_2^4}}
                \end{array}\right)\right\|=
                {2\over\sqrt{x_1^4+x_2^4}}\sqrt{x_1^6+x_2^6},
     \quad\forall x\in -\inte\R^2_+.
$$
Therefore, if setting $x_n=-({1/n},{1/n})\in -\inte\R^2_+$,
one obtains
$$
  \stslK{\nu}(x_n)=2\sqrt{{1\over n^6}+{1\over n^6}\over {1\over n^4}+{1\over n^4}}=
  {2\over n},\quad\forall n\in\N\backslash\{0\}.
$$
Consequently, it results in
$$
  \ssinf=\inf_{x\in -\R^2_+\backslash\{\nullv\}}
  \stslK{\nu}(x)\le
  \inf_{n\in\N\backslash\{0\}} \stslK{\nu}(x_n)=0.
$$
So, hypothesis (iii) is not fulfilled. In spite of this, it
happens that $\Solv\ne\varnothing$. Nevertheless, it is worth
observing that, while a solution to the problem $\SVE$ actually
exists, an error bound such as inequality $(\ref{in:erbossinf})$, with $\ssinf$
replaced with any positive constant $\tau$, fails to work
for the problem at the issue. This because the inequality
$$
  \dist{x}{\Solv}=
  \dist{x}{\{\nullv\}}=\sqrt{x_1^2+x_2^2}\le
  {\sqrt{x_1^4+x_2^4}\over\tau},\quad\forall x\in -\R^2_+,
$$
can never be true, no matter how the value of $\tau>0$ is chosen.
\end{example}

\vskip.5cm


\subsection{Enhanced existence conditions under metric $C$-increase}

Further enhanced existence results for $\SVE$ can be derived by exploiting
conditions able to guarantee hypothesis (iii) of Theorem \ref{thm:SEexist}
to hold ceteris paribus.
As seen in Section \ref{Sect:2}, the property of metric $C$-increase
offers the possibility to estimate the merit function's slope.

\begin{corollary}[Existence under metric increase]   \label{cor:exmetincr}
Under the assumptions (i) and (ii) of Theorem \ref{thm:SEexist},
suppose that:
\begin{itemize}

\item[(iii)] $f(\cdot,z)$ is metrically $C$-increasing on
$K\backslash\Solv$, uniformly in $z\in K$.

\end{itemize}
Then, $\Solv$ is nonempty and closed, and the
following estimate holds true
\begin{equation}
    \dist{x}{\Solv}\le {\nu(x)\over\incr{f}{K\backslash\Solv}-1},
    \quad\forall x\in K.
\end{equation}
\end{corollary}

\begin{proof}
It suffices to observe that the current hypothesis (iii) enables one
to apply Proposition \ref{lem:Cincrstsl}, according to which
one has $\ssinf\ge\incr{f}{K\backslash\Solv}-1>0$. Therefore also
hypothesis (iii) of Theorem \ref{thm:SEexist} is satisfied.
So all the assertions in the thesis follow at once from Theorem
\ref{thm:SEexist}.
\end{proof}

The next step in the present investigation is
to derive from Corollary \ref{cor:exmetincr} verifiable
conditions for the enhanced existence of solutions to $\SVE$,
which rely on differential calculus, with the aim of making the
last result more suitable for applications. This can be done
to a level of generality large enough to include also certain
nonsmooth mappings.

Following \cite{Robi91}, let us say that a mapping $g:\X\longrightarrow\Y$
between Banach spaces is Bouligand-differentiable (for short, B-differentiable)
at $x_0\in\X$ if there exists a mapping $\Bderm{g}{x_0}\in\Ph(\X,\Y)$
(henceforth called the $B$-derivative of $g$ at $x_0$) such that
\begin{equation}   \label{def:Bdiff}
   \lim_{x\to x_0}{g(x)-g(x_0)-\Bder{g}{x_0}{x-x_0}\over \|x-x_0\|}
   =\nullv.
\end{equation}
The reader should notice that, since $\Bderm{g}{x_0}$ is
required to be continuous by the above notion, then $g$ is
continuous at $x_0$ whenever it is B-differentiable at the same point.
Besides, whenever it happens, in particular, that $\Bderm{g}{x_0}
\in\Lin(\X,\Y)\subseteq\Ph(\X,\Y)$, then $g$ turns out to be Fr\'echet differentiable
at $x_0$, with $\Bder{g}{x_0}{v}=\Fder{g}{x_0}{v}$ for every $v\in\X$.
Therefore $B$-differentiability extends Fr\'echet differential
calculus to a broader class of mappings.
A sufficient condition for the metric $C$-increase of $B$-differentiable
mappings can be formulated in terms of existence of directions,
along which the $B$-derivative is ``firmly positive".

\begin{proposition}[Differential condition for metric increase]   \label{pro:Bdifmincr}
Given a mapping $g:\X\longrightarrow\Y$, a convex subset $S\subseteq\X$
and a closed convex cone $C\subseteq\Y$. If $g$ is B-differentiable at each
point of $S\backslash g^{-1}(C)$ and there exists $\sigma>0$
such that for every $x_0\in S\backslash g^{-1}(C)$
\begin{equation}    \label{hyp:Bdifmincrg}
 \exists u_0\in\Usfer\cap\cone(S-x_0):\ \Bder{g}{x_0}{u_0}+
 \sigma\Uball\subseteq C,
\end{equation}
then $g$ is metrically $C$-increasing on $S\backslash g^{-1}(C)$,
with
\begin{equation}     \label{in:mincrdiffest}
   \incr{g}{S\backslash g^{-1}(C)}\ge\sigma+1.
\end{equation}
\end{proposition}

\begin{proof}
Fix arbitrary $\epsilon\in (0,\min\{\sigma,\, 1\})$ and $x_0\in S
\backslash g^{-1}(C)$ and set $\alpha=\sigma+1-\epsilon>1$.
According to $(\ref{def:Bdiff})$, there exists
$\delta>0$ such that
$$
  g(x_0+tv)-g(x_0)-t\Bder{g}{x_0}{v}\in\epsilon t\|v\|\Uball
  \subseteq\epsilon t\Uball,\quad\forall v\in\Uball,\
  \forall t\in [0,\delta].
$$
In particular, taking $u_0\in\Usfer\cap\cone(S-x_0)$ as in
the assumption $(\ref{hyp:Bdifmincrg})$, one finds
\begin{equation}    \label{in:Bdiffgx0}
  g(x_0+tu_0)\in g(x_0)+t\Bder{g}{x_0}{u_0}+\epsilon t\Uball,
  \quad\forall t\in [0,\delta].
\end{equation}
Notice that by convexity of $S$ it holds $x_0+tu_0\in S$ for every
$t\in [0,\delta]$.
Since $g$ is continuous on $S\backslash g^{-1}(C)$ (as a consequence
of its B-differentiability on the same set) and $C$ is
closed, without any loss of generality, one can assume that
$\ball{x_0}{\delta}\cap S\subseteq S\backslash g^{-1}(C)$.
So, by setting $\delta_0=\delta$, fixed any $r\in (0,\delta_0)$,
let us  choose $x_r=x_0+ru_0$. In this way, one has $x_r\in
S\backslash g^{-1}(C)$ and, on account of the inclusions
$(\ref{in:Bdiffgx0})$ and in $(\ref{hyp:Bdifmincrg})$, it
results in
\begin{eqnarray*}
  \ball{g(x_r)}{\alpha r} &=& g(x_r)+\alpha r\Uball \subseteq
  g(x_0)+r\Bder{g}{x_0}{u_0}+\epsilon r\Uball+\alpha r\Uball \\
  &=& g(x_0)+r\Bder{g}{x_0}{u_0}+\epsilon r\Uball+(\sigma+1-\epsilon)r\Uball \\
  &\subseteq& g(x_0)+r[\Bder{g}{x_0}{u_0}+\sigma\Uball]+
  (\epsilon+1-\epsilon)r\Uball \\
  &\subseteq&  g(x_0)+rC+r\Uball\subseteq\ball{g(x_0)+C}{r}.
\end{eqnarray*}
As $x_0$ was arbitrarily chosen in $S\backslash g^{-1}(C)$, the above
inclusion amounts to say that $g$ is metrically $C$-increasing on
$S\backslash g^{-1}(C)$, with $\incr{g}{S\backslash g^{-1}(C)}\ge
\sigma+1-\epsilon$. The estimate in $(\ref{in:mincrdiffest})$
then follows for arbitrariness of $\epsilon$.
\end{proof}

When dealing with bifunctions, the above differentiability notion
will be employed in its (partial) uniform variant: given a bifunction
$f:\X\times\X\longrightarrow\Y$ and a subset $K\subseteq\X$, $f$
is said to be (partially) B-differentiable at $x_0\in\X$, uniformly
in $z\in K$, if there exist mappings $\Bderm{f(\cdot,z)}{x_0}\in\Ph(\X,\Y)$
with the property that for every $\epsilon>0$ there exists $\delta>0$
(depending on $x_0$ and $\epsilon$, but not on $z\in K$) such that
\begin{equation}    \label{in:uniBdiff}
   \sup_{z\in K}{\|f(x,z)-f(x_0,z)-\Bder{f(\cdot,z)}{x_0}{x-x_0}\|
   \over \|x-x_0\|}\le\epsilon,\quad\forall x\in\ball{x_0}{\delta}.
\end{equation}

\begin{example}    \label{ex:Bdiffunif}
(i) Whenever a bifunction $f:\X\times\X\longrightarrow\Y$ takes
the following additively separable form
$$
  f(x,z)=g(x)+h(z),\quad (x,z)\in \X\times K,
$$
where $g:\X\longrightarrow\Y$ is B-differentiable at $x_0$
and $h:K\longrightarrow\Y$, then $f$ is B-differentiable at $x_0$
uniformly in $K$ and it holds $\Bder{f(\cdot,z)}{x_0}{v}=
\Bder{g}{x_0}{v}$, for every $z\in K$ and $v\in\X$.

(ii) Whenever a bifunction $f:\X\times\X\longrightarrow\Y$ takes
the following factorable form
$$
  f(x,z)=\lambda(z)g(x),\quad (x,z)\in \X\times K,
$$
where $g:\X\longrightarrow\Y$ is B-differentiable at $x_0$
and $\lambda:K\longrightarrow\R$ is bounded on $K$, then $f$
is B-differentiable at $x_0$ uniformly in $K$ and it holds
$\Bder{f(\cdot,z)}{x_0}{v}=\lambda(z)\Bder{g}{x_0}{v}$,
for every $z\in K$ and $v\in\X$.
\end{example}

The next technical lemma provides a sufficient condition for the uniform
metric $C$-increase property of uniformly $B$-differentiable bifunctions,
along with an estimate of the exact bound of uniform metric $C$-increase.

\begin{lemma}     \label{lem:unimetincrdiffcon}
With reference to problem $\SVE$, let $f:\X\times\X\longrightarrow\Y$
be a bifunction and let $K\subseteq\X$ be a convex set.
Suppose that $f$ is B-differentiable in the first argument
at each point of $K\backslash\Solv$, uniformly in $z\in K$,
with $B$-derivatives $\Bderm{f(\cdot,z)}{x_0}\in\Ph(\X,\Y)$.
If there exists $\sigma>0$ such that for every $x_0\in K\backslash\Solv$
\begin{equation}    \label{hyp:Bdifmincr}
 \exists u_0\in\Usfer\cap\cone(K-x_0):\
 \Bder{f(\cdot;z)}{x_0}{u_0}+\sigma\Uball\subseteq C,
 \quad\forall z\in K,
\end{equation}
then $f$ is metrically $C$-increasing on $K\backslash\Solv$,
uniformly in $z\in K$, and it holds
\begin{equation}    \label{in:metincrfbest}
   \incr{f}{K\backslash\Solv}\ge\sigma+1.
\end{equation}
\end{lemma}

\begin{proof}
One needs to adapt the proof of Proposition \ref{pro:Bdifmincr}
to the context of uniform B-differentiability for bifunctions.
Fix arbitrary $\epsilon\in (0,\min\{\sigma,\, 1\})$ and $x_0\in K
\backslash\Solv$ and set $\alpha=\sigma+1-\epsilon>1$.
By uniform $B$-differentiability of $f$ at $x_0$, for some $\delta>0$,
taking $u_0\in\Usfer\cap\cone(K-x_0)$ as in $(\ref{hyp:Bdifmincr})$
(not depending on $z\in K$), one has
$$
  f(x_0+tu_0,z)\in f(x_0,z)+t\Bder{f(\cdot;z)}{x_0}{u_0}+
  \epsilon t\Uball,\quad\forall t\in [0,\delta],
  \ \forall z\in K.
$$
Observe that, since
$$
  x_0\in K\backslash\Solv=K\backslash\bigcap_{z\in K}f(\cdot,z)^{-1}(C)
  =\bigcup_{z\in K}[K\backslash f(\cdot,z)^{-1}(C)],
$$
there exists $z_0\in K$ such that $x_0\in K\backslash
f(\cdot,z_0)^{-1}(C)$. As the mapping $x\mapsto f(x,z_0)$ is continuous
at $x_0$, up to a reduction of the value of $\delta$ one finds
$$
  \ball{x_0}{\delta}\cap K\subseteq K\backslash
  f(\cdot,z_0)^{-1}(C)\subseteq K\backslash\Solv.
$$
So, by setting $\delta_0=\delta$ and choosing $x_r=x_0+ru_0$
for any $r\in (0,\delta_0)$, through the same reasoning as in the
proof of Proposition \ref{pro:Bdifmincr}, one obtains
$$
  \ball{f(x_r,z}{\alpha r}\subseteq\ball{f(x_0,z)+C}{r},
  \quad\forall z\in K.
$$
This shows that the conditions in Definition \ref{def:metincrunif}
are satisfied, along with the related estimate, thereby completing
the proof.
\end{proof}

\begin{corollary}    \label{cor:Bdifcondenex}
Under the assumptions (i) and (ii) of Theorem \ref{thm:SEexist},
suppose that:

\begin{itemize}

\item[(iii)] $f$ is B-differentiable in the first argument on
$K\backslash\Solv$, uniformly in $z\in K$, and $K$ is convex;

\item[(iv)] there exists $\sigma>0$ satisfying the condition
in $(\ref{hyp:Bdifmincr})$.

\end{itemize}
Then, $\Solv$ is nonempty and closed, and the
following estimate holds true
\begin{equation}   \label{in:erboBdifcondenex}
    \dist{x}{\Solv}\le {\nu(x)\over\sigma},
    \quad\forall x\in K.
\end{equation}
\end{corollary}

\begin{proof}
It suffices to observe that, under the above assumptions, Lemma
\ref{lem:unimetincrdiffcon} can be applied, so $f$ turns out to
be metrically $C$-increasing on $K\backslash\Solv$ uniformly
in $z\in K$, with the estimate in $(\ref{in:metincrfbest})$ being valid.
In such a circumstance, all its hypotheses being satisfied,
it remains to invoke Corollary \ref{cor:exmetincr} to get
all the assertions in the thesis.
\end{proof}

\begin{example}
Let $\X=\Y=\R^2$ be equipped with its standard Euclidean space structure,
let $f:\R^2\times\R^2\longrightarrow\R^2$ be defined by
$$
  f(x,z)=g(z)-g(x),
$$
where $g(x)=x$, let $\R^2$ be partially ordered by the cone $C=\R^2_+$
and let the constraining set be given by
$$
  K_\theta=\left\{x=(r\cos t,r\sin t)\in\R^2:\ r\ge 0,\ \theta\le t\le
  {\pi\over 2}-\theta\right\},
$$
for any fixed $\theta\in (0,{\pi\over 4})$.
It is clear that the strong equilibrium problem defined by the above dataù
is equivalent to finding the strong efficient solutions to the vector
optimization problem
$$
  \min{}_{\parord} g(x) \qquad\hbox{subject to}\qquad x\in K_\theta.
$$
So one readily sees that $\Solv=\{\nullv\}$.
Since $f$ is continuous over $\R^2\times\R^2$ and $f(\nullv,K_\theta)
\backslash\R^2_+=\varnothing$, the assumptions (i) and (ii) of 
Theorem \ref{thm:SEexist} are obviously satisfied. Since $f$
takes a form such as in Example \ref{ex:Bdiffunif}(i), with $g$
being differentiable, assumption (iii) of Corollary \ref{cor:Bdifcondenex}
is satisfied with
$$
  \Bder{f(\cdot,z)}{x}{v}=-v, \quad\forall x\in K_\theta
  \backslash\{\nullv\}.
$$
Let us check that also assumption (iv) of Corollary \ref{cor:Bdifcondenex}
is satisfied. To this aim, observe that, if taking $x_0\in\inte K_\theta$,
then it is $\cone(K_\theta-x_0)=\R^2$. Therefore, by choosing $u_0=
(-1/\sqrt{2},-1/\sqrt{2})\in\cone(K_\theta-x_0)\cap\Usfer$, one finds
\begin{equation}     \label{in:verassiv1}
  \Bder{f(\cdot,z)}{x_0}{u_0}+{1\over\sqrt{2}}\Uball=
  \left(\begin{array}{c}
     {1\over\sqrt{2}} \\
     {1\over\sqrt{2}}
  \end{array}\right)+{1\over\sqrt{2}}\Uball\subseteq\R^2_+.
\end{equation}
Now take $x_0=(r_0\cos\theta,r_0\sin\theta)\in K_\theta\backslash
\{\nullv\}$, for any $r_0>0$. In this case, one has
$$
   \cone(K_\theta-x_0)=\{x=(r\cos t,r\sin t)\in\R^2:\ 
   r\ge 0,\ \theta\le t\le\theta +\pi\}.
$$
Thus, if taking $u_0=(\cos(\theta +\pi),\sin(\theta +\pi))=-
(\cos\theta,\sin\theta)$, one finds
\begin{equation}    \label{in:verassiv2}
  \Bder{f(\cdot,z)}{x_0}{u_0}+\sin\theta\,\Uball=
  \left(\begin{array}{c}
     \cos\theta \\
     \sin\theta
  \end{array}\right)+\sin\theta\,\Uball\subseteq\R^2_+.
\end{equation}
Notice that the last inclusion is true because it is $\sin\theta<\cos\theta$,
as $\theta\in (0,{\pi\over 4})$.
The remaining case $x_0=\left(r_0\cos\left({\pi\over 2}-\theta\right),
r_0\sin\left({\pi\over 2}-\theta\right)\right)\in K_\theta\backslash
\{\nullv\}$, for any $r_0>0$, can be discussed in a similar manner by
taking into account the symmetry of $K_\theta$ with respect to the axe
of equation $x_2=x_1$. Thus, since it is $\sin\theta<{1\over\sqrt{2}}$,
by inclusions $(\ref{in:verassiv1})$ and $(\ref{in:verassiv2})$ one
verifies that condition $(\ref{hyp:Bdifmincr})$ is fulfilled with
$\sigma=\sin\theta$. As a consequence, Corollary \ref{hyp:Bdifmincr}
can be applied to the problem under examination.

In order to check the validity of the error bound provided in its thesis,
one needs to find the expression of $\nu$. This is easily done,
inasmuch as one has
\begin{eqnarray*}
    \nu(x) &=& \sup_{z\in K_\theta}\dist{z-x}{\R^2_+}=
    \sup_{z\in K_\theta}\dist{z}{x+\R^2_+} \\
    &\ge & \dist{\nullv}{x+\R^2_+}=\|x\|,
    \quad\forall x\in K_\theta.
\end{eqnarray*}
On the other hand, as $K_\theta\subseteq\R^2_+$, one gets
\begin{eqnarray*}
    \nu(x) &\le& \sup_{z\in\R^2_+}\dist{z}{x+\R^2_+}=
    \dist{\nullv}{x+\R^2_+}=\|x\|,  \quad\forall x\in K_\theta.
\end{eqnarray*}
From the last inequalities, one obtains
$$
  \nu(x)=\|x\|,  \quad\forall x\in K_\theta.
$$ 
So, in the light of the above computations, it is possible
to check that it actually holds
$$
  \dist{x}{\Solv}=\dist{x}{\{\nullv\}}=\|x\|\le
  {\|x\|\over \sin\theta}={\nu(x)\over\sigma},\quad
  \forall x\in K_\theta,
$$
in accordance with inequality $(\ref{in:erboBdifcondenex})$.

The reader should be warned that the same conclusions can not be
drawn in the (critical) case with $\theta=0$, corresponding
to a vector equilibrium problem defined by the same bifunction $f$,
but with $K=\R^2_+$. There is no $\sigma>0$ for which assumption
(iv) of Corollary \ref{cor:Bdifcondenex} happens to be satisfied.
Indeed, if taking $x_0=(r_0,0)\in K\backslash\{\nullv\}$,
for some $r_0>0$, one can not find $u_0\in\cone(\R^2_+-(r_0,0))=
\{x=(x_1,x_2)\in\R^2_+:\ x_2\ge 0\}$ such that
$$
  -u_0+\sigma\Uball\subseteq\R^2_+.
$$
In spite of this, it is still $\nu(x)=\|x\|$ for every $x\in\R^2_+$
and $\Solv=\{\nullv\}$, so the error bound
$$
  \dist{x}{\Solv}=\|x\|\le\nu(x),\quad\forall x\in\R^2_+
$$
still holds true. This fact shows once more that the conditions
for the enhanced existence established by the present approach
can be only sufficient.
\end{example}

\vskip.5cm


\subsection{Subdifferential conditions for the enhanced solution existence}

Throughout the current subsection the space $(\X\,\|\cdot\|)$
is assumed to be reflexive.

In order to formulate the next enhanced existence results for
problem $\SVE$, some more technical tools from nonsmooth analysis
need to be recalled.
Let $K\subseteq\X$ be a nonempty closed subset and let
$x\in\X\backslash K$ such that $\Proj{x}{K}\ne\varnothing$. In
\cite[Proposition 1.102]{Mord06a} it has been proved that
$$    \label{in:Fsubddistout}
  \Fsub\dist{\cdot}{K}(x)\subseteq\bigcap_{w\in\Proj{x}{K}}
  \FNcone{w}{K}\cap\Usfer^*,
$$
where
$$
  \Fsub\varphi(x)=\left\{x^*\in\X^*:\ \liminf_{v\to\nullv}
  {\varphi(x+v)-\varphi(x)-\langle x^*,v\rangle\over \|v\|}
  \ge 0\right\}
$$
denotes the (regular) Fr\'echet subdifferential
of a function $\varphi$ at $x\in\dom\varphi$, and
$$
  \FNcone{w}{K}=\left\{x^*\in\X^*:\ \limsup_{x\stackrel{K}{\to} w}
  {\langle x^*,x-w\rangle\over \|x-w\|}\le 0\right\}
$$
denotes the cone of the Fr\'echet normals (a.k.a. prenormal cone) to $K$ at $w$.
On the other hand, if $x\in K$ then \cite[Corollary 1.96]{Mord06a}
provides the different representation
$$
  \Fsub\dist{\cdot}{K}(x)=\FNcone{x}{K}\cap\Uball^*.
$$
Since in what follows both the cases have to be considered, it is
convenient to deal with the set-valued mapping $\widehat{B}^*_K:
\X\rightrightarrows\X^*$, defined by
\begin{eqnarray}    \label{def:distFsubd}
   \widehat{B}^*_K(x)=\left\{\begin{array}{ll}
    \FNcone{x}{K}\cap\Uball^* & \hbox{ if } x\in K, \\
    \\
    \displaystyle\bigcap_{w\in\Proj{x}{K}}\FNcone{w}{K}\cap\Usfer^* & \hbox{ if }
    x\not\in K.
   \end{array}\right.
\end{eqnarray}
Whenever $K$ is a closed convex set, in a reflexive Banach space setting
it is $\Proj{x}{K}\ne\varnothing$ for every $x\in\X$. Thus,
since in such an event it turns out that
$\Fsub\dist{\cdot}{K}(x)=\partial\dist{\cdot}{K}(x)$ and
$\FNcone{w}{K}=\Ncone{w}{K}$ (see \cite[Theorem 1.93]{Mord06a}
and \cite[Proposition 1.3]{Mord06a}, respectively), the set-valued mapping
$\widehat{B}^*_K$ is well-defined and takes the special form
\begin{eqnarray*}
   \widehat{B}^*_K(x)=\left\{\begin{array}{ll}
    \Ncone{x}{K}\cap\Uball^* & \hbox{ if } x\in K, \\
    \\
    \displaystyle\bigcap_{w\in\Proj{x}{K}}\Ncone{w}{K}\cap\Usfer^* & \hbox{ if }
    x\not\in K.
   \end{array}\right.
\end{eqnarray*}

\begin{remark}
It is well known that
if $x\in\X\backslash K$ and  $(\X,\|\cdot\|)$ has a uniformly
G\^ateaux differentiable norm, then function $\dist{\cdot}{K}$ is strictly
differentiable at $x$, with $\Fderm{\dist{\cdot}{K}}{x}\in\Usfer^*$
(see \cite[Corollary 4.2]{WuYe03}), so in this case one has
$\widehat{B}^*_K(x)=\{\Fderm{\dist{\cdot}{K}}{x}\}$. The class
of Banach spaces admitting a uniformly G\^ateaux differentiable norm
is known to include all separable spaces (see \cite[Chapter II, Corollary 6.9(i)]{DeGoZi93}).
\end{remark}

\begin{theorem}[$C$-concave case]    \label{thm:Cconcenexsubd}
With reference to a problem $\SVE$, suppose that:
\begin{itemize}

\item[(i)] each function $x\mapsto f(x,z)$ is $C$-concave on
the convex set $K$, for every $z\in K$;

\item[(ii)] there exists $x_0\in K$ such that $f(x_0,K)$
is $C$-bounded;

\item[(iii)] each function $x\mapsto f(x,z)$ is $C$-u.s.c.
on $\X$, for every $z\in K$;

\item[(iv)] there exists $\gamma>0$ such that
\begin{equation}    \label{cap:subdfarcond}
  \left[\partial\nu(x)+\widehat{B}^*_K(x)\right]\cap
  \gamma\Uball^*=\varnothing,\quad\forall x\in\X
  \backslash\Solv.
\end{equation}
\end{itemize}
Then, $\Solv$ is nonempty, closed and convex, and the
following estimate holds true
\begin{equation}    \label{in:erbosubsl}
    \dist{x}{\Solv}\le {\nuka(x)\over\gamma},
    \quad\forall x\in\X.
\end{equation}
\end{theorem}

\begin{proof}
Observe first that, on account of Remark \ref{rem:nulsc}
and hypothesis (iii), the merit function $\nuka$ is l.s.c. on
$\X$. Moreover, as stated in Remark \ref{rem:conconvnuka},
under hypotheses (i) $\nuka$ is convex on $\X$ and $x_0\in\dom\nuka$,
so that $[\nuka<+\infty]\ne\varnothing$.
Now, take $x\in [0<\nuka<+\infty]$.
Since $\nuka$ is proper, convex and l.s.c. on $\X$, then according
to Remark \ref{rem:stslstslK}(iii), the following slope estimate holds
$$
  \stsl{\nuka}(x)=\dist{\nullv^*}{\partial\nuka(x)}.
$$
Since both $\nu$ and $\dist{\cdot}{K}$ are convex functions, while $\nu$ is l.s.c.
and $\dist{\cdot}{K}$ is continuous at $x_0\in\dom\nu\cap\dom\dist{\cdot}{K}$,
by the Moreau-Rockafellar theorem one obtains
$$
  \partial\nuka(x)=\partial\nu(x)+\partial\dist{\cdot}{K}(x)
  \subseteq \partial\nu(x)+\widehat{B}^*_K(x).
$$
Thus, the condition in hypothesis (iv) implies
$$
  \stsl{\nuka}(x)=\dist{\nullv^*}{\partial\nuka(x)}\ge\gamma,
  \quad\forall x\in [0<\nuka<+\infty].
$$
This estimate shows that is possible to apply Proposition \ref{pro:Aze03}
with $X=\X$ and $\varphi=\nuka$, what leads to obtain the
nonemptiness of $\Solv$ and the inequality
$$
  \dist{x}{\Solv}\le {\nuka(x)\over\gamma},
    \quad\forall x\in [\nuka<+\infty].
$$
The last inequality readily allows one to achieve the estimate in
$(\ref{in:erbosubsl})$.
The convexity of $\Solv$ is a direct consequence of the convexity
of $\nuka$, as observed in Remark \ref{rem:funchar}. Thus the
proof is complete.
\end{proof}

\begin{remark}
It should be noticed that inequality $(\ref{in:erbosubsl})$, in particular,
entails
$$
   \dist{x}{\Solv}\le {\nu(x)\over\gamma},
   \quad\forall x\in K.
$$
\end{remark}

As a comment to Theorem \ref{thm:Cconcenexsubd}, it is to
be pointed out that the condition in hypothesis (iv), involving
function $\nu$, is not explicitly formulated in terms of problem data.
In fact, expressing the subdifferential of $\nu$ in terms of
generalized derivatives of $f$ may imply nontrivial
calculations.
Nevertheless, in the special case in which $K$ is compact
and each function $x\mapsto f(x,z)$ is smooth,
with surjective derivatives, this issue can be faced by
means of well-known subdifferential calculus rules and
other technical results in nonsmooth analysis. To see this
in detail, given $x\in K$ let us set
$$
  \overline{K}_x=\{z\in K:\ \dist{f(x,z)}{C}=\nu(x)\}
$$
and let us introduce the set-valued mapping $B^*_C:\X\rightrightarrows\Y^*$,
defined by
\begin{eqnarray*}
   B^*_C(x)=\left\{\begin{array}{ll}
    \ndc{C}\cap\Uball^* & \hbox{ if } \nu(x)=0, \\
    \\
    \ndc{C}\cap\Usfer^* & \hbox{ if } \nu(x)>0.
   \end{array}\right.
\end{eqnarray*}

\begin{proposition}
With reference to a problem $\SVE$, suppose that:
\begin{itemize}

\item[(i)] $K$ is compact and convex and $(\Y,\|\cdot\|)$ is
reflexive;

\item[(ii)] $f:\X\times\X\longrightarrow\Y$ is
continuous on $\X\times\X$;

\item[(iii)] each function $x\mapsto f(x,z)$ is $C$-concave
on $\X$, for every $z\in K$;

\item[(iv)] each function $x\mapsto f(x,z)$ is ${\rm C}^1(\X)$,
with $\Fderm{f(\cdot,z)}{x}\in\Lin(\X,\Y)$ being onto,
for every $z\in K$;

\item[(v)] there exists $\gamma>0$ such that
\begin{eqnarray*}
  \left[\wclco\left(\bigcup_{z\in \overline{K}_x}
  \Fderm{f(\cdot,z)}{x}^*(B^*_C(x))\right)
  +\widehat{B}^*_K(x)\right]\cap\gamma\Uball^*=
  \varnothing,  \\
  \forall x\in\X\backslash\Solv.
\end{eqnarray*}
\end{itemize}
Then, $\Solv$ is nonempty, closed and convex, and inequality
$(\ref{in:erbosubsl})$ holds true.
\end{proposition}

\begin{proof}
Let us start with observing that, by virtue of hypothesis (ii),
the bifunction $\dC\circ f$ is continuous on $\X\times\X$.
By the compactness of $K$,
this implies that, for every $x\in K$, the value $\nu(x)=
\sup_{z\in K}\dist{f(x,z)}{C}$ is attained at some $z\in K$,
so that $\overline{K}_x\ne\varnothing$ for every $x\in K$.
Furthermore, the continuity of $\dC\circ f$ entails that
each function $x\mapsto\dist{f(x,z)}{C}$ is continuous on $\X$
and each function $z\mapsto\dist{f(x,z)}{C}$ is, in particular,
u.s.c. on $\X$. Since by virtue of hypothesis (iii) each
function $x\mapsto\dist{f(x,z)}{C}$ is convex on $\X$
(remember Lemma \ref{lem:Cconcconv}), these facts enable one
to apply the well-known max rule for the subdifferential of
a supremum of convex functions (see, for instance,
\cite[Theorem 2.4.18]{Zali02}), which gives
\begin{eqnarray}  \label{in:wclcomaxform}
    \partial\nu(x) &=& \partial\left(\sup_{z\in K}
    \dist{f(\cdot,z)}{C}\right)(x) \nonumber \\
    &=& \wclco\left(\bigcup_{z\in \overline{K}_x}
    \partial\dist{f(\cdot,z)}{C}(x)\right).
\end{eqnarray}
Since each function $x\mapsto(x,z)$ is,
in particular, strictly differentiable and its derivative
is surjective according to hypothesis (iv), then it is possible
to employ the formula in \cite[Proposition 1.112(i)]{Mord06a},
which rules the subdifferential under composition with
smooth mappings. In doing so, by recalling that the Mordukhovich
(a.k.a. basic or limiting) subdifferential
coincides here with the subdifferential in the sense of convex
analysis by the convexity of the involved functions, one finds
$$
  \partial\dist{f(\cdot,z)}{C}(x)\subseteq
  \Fderm{f(\cdot,z)}{x}^*\partial\dC(f(x,z)).
$$
Now, observe that if $x\in\X\backslash\Solv$ and $\nu(x)>0$,
then for every $z\in\overline{K}_x$ it must be $f(x,z)\not\in C$.
Consequently, as also $(\Y,\|\cdot\|)$ is reflexive,
$\Proj{f(x,z)}{C}\ne\varnothing$ and therefore
$$
   \partial\dC(f(x,z))\subseteq \bigcap_{w\in\Proj{f(x,z)}{C}}
   \Ncone{w}{C}\cap\Usfer^*\subseteq\ndc{C}\cap\Usfer^*.
$$
If $x\in\X\backslash\Solv$ but $\nu(x)=0$, then it happens that
$f(x,z)\in C$ also for $z\in\overline{K}_x$, so one can only
say
$$
  \partial\dC(f(x,z))\subseteq\ndc{C}\cap\Uball^*.
$$
Thus, from inclusion $(\ref{in:wclcomaxform})$ one obtains in any case
$$
  \partial\nu(x)\subseteq \wclco\left(\bigcup_{z\in \overline{K}_x}
  \Fderm{f(\cdot,z)}{x}^*(B^*_C(x))\right),
  \quad\forall x\in [\nuka>0].
$$
In the light of the last inclusion, it is clear that hypothesis (v)
implies the validity of condition (iv) in Theorem \ref{thm:Cconcenexsubd}.

It remains to notice that, for any fixed $x\in\X$, by continuity of the
function $z\mapsto\dist{f(x,z)}{C}$ and by compactness of $K$,
also the set $f(x,K)$ is compact and hence $C$-bounded, so
hypothesis (ii) of Theorem \ref{thm:Cconcenexsubd} is fulfilled.
As for hypothesis (iii) of Theorem \ref{thm:Cconcenexsubd}, it comes
true as an easy consequence of the current hypothesis (ii).
Thus, the thesis can be achieved by applying Theorem \ref{thm:Cconcenexsubd}.
\end{proof}

Another worthwhile comment refers to condition $(\ref{cap:subdfarcond})$,
which, as a requirement for subgradients to be sufficiently away from $\nullv^*$,
can be regarded as a regularity condition. The reader should notice that if
$x\in [\nuka>0]\cap K$, then $\nuka(x)=\nu(x)$, so it must $\nu(x)>0$.
If $K$ is such that $\nu(x)>\dist{f(x,z)}{C}$ for at least some
$z\in K$, that is  $\dist{f(x,\cdot)}{C}$ is not constant over $z\in K$,
then $x$ cannot be a minimizer of $\nu$, with the consequence that
$\nullv^*\not\in\partial\nu(x)$. Therefore, if $\partial\nu(x)$ lies
sufficiently faraway from the origin, it may actually happen
that $\left[\partial\nu(x)+\widehat{B}^*_K(x)\right]\cap
\gamma\Uball^*=\varnothing$.
On the other hand, if $x\in [\nuka>0]\backslash K$, then it may happen
that $\nu(x)=0$, so $\nullv^*\in\partial\nu(x)$, but in such an event,
according to $(\ref{def:distFsubd})$, it
is $\nullv^*\not\in\widehat{B}^*_K(x)$. Thus, again the condition
$(\ref{cap:subdfarcond})$ may actually take place.

By replacing the Fenchel subdifferential with more involved
tools of nonsmooth analysis, the above line of investigation can be expanded
in such a way to consider also problems without $C$-concave
bifunctions.
Below, the reader will find an attempt to develop the
analysis by employing Fr\'echet and Mordukhovich subgradients.
Let $\varphi:\X\longrightarrow\R\cup\{\pm\infty\}$ be function,
l.s.c. around a point $x\in\dom\varphi$ and let $(\X,\|\cdot\|)$
be an Asplund space. In this setting, an equivalent way to introduce
the Mordukhovich subdifferential $\Msub\varphi(x)$ of $\varphi$ at
$\bar x$ by using the Fr\'echet subdifferential is to define
$$
  \begin{array}{ccc}
  \Msub\varphi(\bar x)= & \Uplim & \Fsub\varphi(x), \\
  & \hbox{\scriptsize $x\stackrel{\varphi}{\to} \bar x$} &
  \end{array}
$$
where $\Uplim_{x\stackrel{\varphi}{\to} \bar x}$ denotes the
Painlev\'e-Kuratowski upper limit of the set-valued mapping
$\Fsub\varphi:\X\rightrightarrows\X^*$ as $x\to\bar x$ and
$\varphi(x)\to\varphi(\bar x)$, with respect to the norm topology
on $\X$ and the \weakstar topology on $\X^*$ (see, for more
details, \cite{Mord06a} and Theorem 2.34 therein).

\begin{theorem}[Mordukhovich subdifferential condition]    \label{thm:enexistFsub}
With reference to a problem $\SVE$, suppose that:
\begin{itemize}

\item[(i)] the set $K$ is convex;

\item[(ii)] each function $x\mapsto f(x,z)$ is $C$-u.s.c.
on $\X$, for every $z\in K$;

\item[(iii)] there exists $x_0\in K$ such that $f(x_0,K)$
is $C$-bounded;

\item[(iv)] there exists $\gamma>0$ such that
\begin{equation}    \label{cap:Msubdfarcond}
  \left[\Msub\nu(x)+\widehat{B}^*_K(x)\right]\cap
  \gamma\Uball^*=\varnothing,\quad\forall
  x\in\X\backslash\Solv.
\end{equation}
\end{itemize}
Then, $\Solv$ is nonempty and closed, and the
the estimate in $(\ref{in:erbosubsl})$ holds true.
\end{theorem}

\begin{proof}
By hypothesis (ii) the merit function $\nuka:\X\longrightarrow
[0,+\infty]$ is l.s.c. on $\X$ and therefore the set
$[\nuka>0]$ is open. By hypothesis (iii) it is $x_0\in
[\nuka<+\infty]\ne\varnothing$. Moreover, observe that
$(\X,\|\cdot\|)$, as a reflexive Banach space, is an Asplund
space. According to \cite[Theorem 1, Chapter 2]{Ioff00},
this is equivalent to the fact that $(\X,\|\cdot\|)$ is
$\Fsub$-trustworthy. Such a property allows one to exploit
the slope estimate in \cite[Proposition 1, Chapter 3]{Ioff00}
valid for l.s.c. functions on open subsets of $\Fsub$-trustworthy
Banach spaces, according to which
\begin{equation}    \label{in:stslFsubdest}
   \inf_{x\in [\nuka>0]}\stsl{\nuka}(x)\ge
   \inf_{x\in [\nuka>0]}\dist{\nullv^*}{\Fsub\nuka(x)}.
\end{equation}
By recalling that for any $x\in\X$ the following general
relation between subdifferentials holds
$$
  \Fsub\nuka(x)\subseteq\Msub\nuka(x),
$$
from inequality $(\ref{in:stslFsubdest})$ one obtains
\begin{equation}
   \inf_{x\in [\nuka>0]}\stsl{\nuka}(x)\ge
   \inf_{x\in [\nuka>0]}\dist{\nullv^*}{\Msub\nuka(x)}.
\end{equation}
Now, as $\nu$ and $\dist{\cdot}{K}$ are l.s.c. and Lipschitz
continuous on $\X$, respectively, they form a semi-Lipschitzian
sum at any $x\in\dom\nuka$ in the sense of \cite[Chapter 2.4]{Mord06a}.
Thus, as $(\X,\|\cdot\|)$ is an Asplund space, according
to the sum rule for the Mordukhovich subdifferential
(see \cite[Theorem 2.33(c)]{Mord06a}) one finds
$$
  \Msub\nuka(x)\subseteq\Msub\nu(x)+\Msub\dist{\cdot}{K}(x)
  \subseteq\Msub\nu(x)+\widehat{B}^*_K(x),
$$
where the last inclusion holds because the Mordukhovich
subdifferential coincides with the Fenchel subdifferential,
in consideration of the convexity of $K$.
By virtue of hypothesis (iv), the last inclusion implies
$$
  \inf_{x\in [\nuka>0]}\stsl{\nuka}(x)\ge\gamma.
$$
Such an inequality enables one to apply Proposition \ref{pro:Aze03},
from which all the assertions in the thesis can be deduced.
This completes the proof.
\end{proof}

\begin{remark}
The author is aware of the fact that Theorem \ref{thm:enexistFsub}
can be subsumed in a more general scheme of analysis, where the Fr\'echet
and the Mordukhovich subdifferentials are replaced with any subdifferential
$\partial_\Box$, axiomatically defined as in \cite[Section 1.5, Chapter 2]{Ioff00},
and $(\X,\|\cdot\|)$ is supposed to be $\partial_\Box$-trustworthy, as meant
in \cite[Definition 4, Chapter 2]{Ioff00}).
It is clear that, following such an approach, the subdifferential condition
$(\ref{cap:Msubdfarcond})$ should be expected to take a more involved form,
because of the need of employing fuzzy sum rules for the $\partial_\Box$
subdifferential.
\end{remark}



\section{Conclusions}    \label{Sect:4}

The contents of the present paper describe an attempt to conduct a study
of solvability and error bounds for vector equilibrium problems
via a variational approach. This attempt leads to obtain, as a main
result, conditions which are expressed in terms of $B$-derivatives,
convex normals and subgradients as well as Mordukhovich subdifferential.
With respect to results focussing exclusively on the existence issue
and obtained by different approaches, the conditions here established
clearly demonstrate the crucial role that nonsmooth analysis may play
in the development of this area. The achievements here discussed
leave open the possibility for subsequent refinements and improvements.
Some of them should come from a convenient expressions of subdifferentials
of $\nu$ in terms of proper generalized derivatives of the bifunction $f$.




\vskip1cm

\begin{thebibliography}{99}


\bibitem{AnKoYa01} Ansari Q.H., Konnov  I. V., and  Yao J. C.:
{\it  Existence of a solution and variational principles for vector
equilibrium problems}, J. Optim. Theory Appl. \textbf{110} (2001),
no. 3, 481--492.

\bibitem{AnOeSc96} Ansari Q.H., Oettli W., and Schl\"ager D:
{\it A generalization of vectorial equilibria}, In: Operations
Research and Its Applications, World Publishing Corporation,
181--185, 1996.

\bibitem{AnOeSc97} Ansari Q.H., Oettli W., and Schl\"ager D:
{\it A generalization of vectorial equilibria}, Math. Methods Oper.
Res. \textbf{46} (1997), no. 2, 147--152.

\bibitem{Ansa00} Ansari Q.H.:
{\it Vector equilibrium problems and vector variational inequalities},
in Vector variational inequalities and vector equilibria, 1--15,
Nonconvex Optim. Appl., \textbf{38}, Kluwer Acad. Publ., Dordrecht, 2000.

\bibitem{Aze03} Az\'e D.:
{\it A survey on error bounds for lower semicontinuous functions},
ESAIM Proc., \textbf{13} (2003), 1--17.

\bibitem{AzeCor04} Az\'e D. and Corvellec J.-N.:
{\it Characterizations of error bounds for lower semicontinuous functions
on metric spaces}, ESAIM Control Optim. Calc. Var. \textbf{10} (2004),
no. 3, 409--425.

\bibitem{BiHaSc97} Bianchi M., Hadjisavvas N., and Schaible S.:
{\it Vector equilibrium problems with generalized monotone bifunctions},
J. Optim. Theory Appl. \textbf{92} (1997), no. 3, 527--542.

\bibitem{BiCaPaPa13} Bigi G., Castellani M., Pappalardo M., and
Passacantando M.:
{\it Existence and solution methods for equilibria},
European J. Oper. Res. \textbf{227} (2013), no. 1, 1--11. 

\bibitem{BluOet94} Blum E. and Oettli W.,
{\it  From optimization and variational inequalities to equilibrium
problems}, Math. Student \textbf{63} (1994), no. 1-4, 123--145.

\bibitem{BorZhu05} Borwein J.M. and Zhu Q. J.: {\it Techniques of variational
analysis}, Springer-Verlag, New York, 2005.

\bibitem{CeMaYa08} Ceng L.C., Matroeni G., and Yao J.C.:
{\it Existence of solutions and variational principles for generalized vector
systems}, J. Optim. Theory Appl. \textbf{137} (2008), no. 3, 485--495.

\bibitem{DeGoZi93} Deville R., Godefroy G., and Zizler V.:
{\it  Smoothness and renormings in Banach spaces},
Longman Scientific \& Technical, Harlow; copublished with John Wiley \& Sons,
Inc., New York, 1993.

\bibitem{DeMaTo80} De Giorgi E., Marino A., and Tosques M.: {\it Problems
of evolution in metric spaces and maximal decreasing curve}, Atti Accad.
Naz. Lincei Rend. Cl. Sci. Fis. Mat. Natur. (8) \textbf{68} (1980), 180--187.

\bibitem{FaHeKrOu10} Fabian M.J., Henrion R., Kruger A.Y., and Outrata J.V.:
{\it Error bounds: necessary and sufficient conditions}, Set-Valued Var.
Anal. \textbf{18} (2010), no. 2, 121--149.

\bibitem{FreKas99} Frenk J.B. and Kassay G.:
{\it On classes of generalized convex functions, Gordan-Farkas type
theorems, and Lagrangian duality}. J. Optim. Theory Appl. \textbf{102}
(1999), no. 2, 315--343.

\bibitem{Gong06} Gong X.H.:
{\it Strong vector equilibrium problems}, J. Global Optim.
\textbf{36} (2006), no. 3, 339--349.

\bibitem{Ioff00} Ioffe A.D.:
{\it Metric regularity and subdifferential calculus}, Russian Math.
Surveys \textbf{55} (2000), no. 3, 501--558.

\bibitem{KanAki82} Kantorovich L.V. and Akilov G.P.:
{\it Functional analysis}, Second edition, Pergamon Press,
Oxford-Elmsford, N.Y., 1982.

\bibitem{Luc89} Luc D.T.:
{\it Theory of vector optimization}, Springer-Verlag, Berlin, 1989.

\bibitem{Mord06a} Mordukhovich B.S.:
{\it Variational analysis and generalized differentiation. I. Basic Theory}.
Springer-Verlag, Berlin, 2006.

\bibitem{Mord06b} Mordukhovich B.S.:
{\it Variational analysis and generalized differentiation. II. Applications}.
Springer-Verlag, Berlin, 2006.

\bibitem{Peno13} Penot J.-P.:
{\it Calculus without derivatives}, Springer, New York, 2013.

\bibitem{Robi91} Robinson S.M.:
{\it An implicit-function theorem for a class of nonsmooth functions}.
Math. Oper. Res. \textbf{16} (1991), no. 2, 292--309.

\bibitem{Rubi77} Rubinov A.M.:
{\it Sublinear operators and their applications}, Uspehi Mat. Nauk
\textbf{32} (1977), no. 4(196), 113--174. [Russian]

\bibitem{Uder19} Uderzo A.:
{\it On some generalized equations with metrically $C$-increasing mappings:
solvability and error bounds with applications to optimization},
Optimization \textbf{68} (2019), no. 1, 227--253.

\bibitem{Uder20} Uderzo A.:
{\it Solution analysis for a class of set-inclusive generalized equations:
a convex analysis approach}, Pure Appl. Funct. Anal. \textbf{5} (2020),
no. 3, 769-–790.

\bibitem{Uder21} Uderzo A.:
{\it On the quantitative solution stability of parameterized set-valued inclusions},
Set-Valued Var. Anal. \textbf{29} (2021), no. 2, 425--451.

\bibitem{Uder22} Uderzo A.:
{\it On tangential approximations of the solution set of set-valued inclusions},
to appear on J. Appl. Anal., 1--23,  https://doi.org/10.1515/jaa-2021-2049.

\bibitem{WuYe03} Wu Z. and Ye J.J.:
{\it Equivalence among various derivatives and subdifferentials of the distance
function}, J. Math. Anal. Appl. \textbf{282} (2003), no. 2, 629--647.

\bibitem{Zali02} Z\u alinescu C.: {\it Convex analysis in general vector spaces},
World Scientific Publishing Co., River Edge, NJ, 2002.

\end{thebibliography}
\end{document}